\documentclass[12pt,makeidx]{amsart}

\usepackage{amssymb}
\usepackage{txfonts}

\usepackage{enumerate,color}
\usepackage[pagebackref,colorlinks,linkcolor=red,citecolor=blue,urlcolor=blue,hypertexnames=true]{hyperref}
\usepackage{url}
\usepackage{amsmath,amscd}

\renewcommand{\subset}{\subseteq}
\newtheorem{theorem}            {Theorem}[section]
\newtheorem{corollary}          [theorem]{Corollary}

\newtheorem{definition}         [theorem]{Definition}
\newtheorem{lemma}              [theorem]{Lemma}
\newtheorem{remark}         [theorem]{Remark}

\newcommand{\df}[1]{{\it{#1}}{\index{#1}}}

\makeindex

\begin{document}

\setcounter{page}{1}
\title[Multipliers and Dirichlet Series]{Multipliers on Hilbert Spaces of Dirichlet Series}

\author[Stetler]{Eric Stetler}
\address{Department of Mathematics\\
  University of Florida 
   }
   \email{estetler@ufl.edu}


\date{\today}

\maketitle

\begin{abstract}
	In this paper, certain classes of Hilbert spaces of Dirichlet series with weighted norms and their corresponding multiplier algebras will be explored. For a sequence $\{w_n\}_{n=n_0}^\infty $ of positive numbers, define
		\[\mathcal H^\textbf{w}=\left\{\sum_{n=n_0}^\infty  a_nn^{-s}:\sum_{n=n_0}^\infty |a_n|^2 w_n<\infty\right\}.\]
	Hedenmalm, Lindqvist and Seip considered the case in which $w_n\equiv 1$ and classified the multiplier algebra of $\mathcal H^\textbf{w}$ for this space in \cite {HLS}. In \cite{M}, McCarthy classified the multipliers on $\mathcal H^\textbf{w}$ when the 
	weights are given by
		\[w_n=\int_0^\infty  n^{-2\sigma}d\mu(\sigma),\]
	where $\mu$ is a positive Radon measure with $\{0\}$ in its support and $n_0$ is the smallest positive integer for which this integral is finite. Similar results will be derived assuming the weights 
	are  multiplicative, rather than given by a measure. In particular, upper and lower bounds on the operator norms of the multipliers will be obtained, in terms of their values on certain half planes, on the Hilbert spaces resulting from these weights. Finally, 
	some number theoretic weight sequences will be explored and the multiplier algebras of the corresponding Hilbert spaces determined up to isometric isomorphism, providing examples where the conclusion of McCarthy's result holds, but  under alternate 
	hypotheses on the weights. 
\end{abstract}

\section{Introduction}
\label{sec:intro}

	A \df{Dirichlet series} is a series of the form
		\begin{equation}
			\label{eq:Dirichlet definition}
				\sum_{n=1}^\infty  a_n n^{-s},
		\end{equation}
	where the $a_n$'s and $s$  are complex numbers. Such a series may or may not converge, depending on the $a_n$'s and the choice for $s$. For example,  if $a_n=n!$ for each $n$, then the series fails to
	converge anywhere. If instead $a_n$ is nonzero for only finitely many $n$, then the series converges everywhere. It turns out that if a Dirichlet series converges for some complex number $s_{0}$, then it
	must converge for all complex $s$ with real part, $\mathfrak{Re}(s)$, strictly larger than $s_{0}$. Given a real number $\delta$, let $\Omega_{\delta}$ denote the open half-plane
  		\[\left\{z\in\mathbb C:\mathfrak{Re}(z)>\delta \right\}.\]
	By the preceding remark, if a Dirichlet series converges at $s_{0}$, then it converges in $\Omega_{\mathfrak{Re}(s_{0})}$.\\

	We  are going to be concerning ourselves with Hilbert spaces of functions representable by Dirichlet series in some open half-plane. Denote by $\mathcal D$ the set of those functions whose domain contains some open right half-plane (which right 
	half-plane will generally depend on the function) in which the function is representable by a Dirichlet series.  Let  $\mathcal D_\delta$ denote the set of functions that can be represented by a Dirichlet series specifically in the right half-plane 
	$\Omega_\delta$.

	Let  $\{w_n\}_{n=n_0}^\infty $ be a sequence of positive real numbers and suppose that there is a positive real number $\sigma_0$ such that 
		\begin{equation}
			\label{eq:weight condition}
				\sum_{n=n_0}^\infty  w_n^{-1}n^{-2\sigma}<\infty
		\end{equation}
	whenever $\sigma>\sigma_0$. Let $\mathcal H^\textbf{w}$ denote the Hilbert space
		\begin{equation}
			\label{eq:H^w definition}
				\left\{\sum_{n=n_0}^\infty  a_n n^{-s}:\sum_{n=n_0}^\infty |a_n|^2 w_n<\infty\right\},
		\end{equation}
	with inner product $\langle\cdot,\cdot\rangle_\textbf{w}$ defined by 
		\[\langle \sum_{n=n_0}^\infty  a_n n^{-s},\sum_{m=n_0}^\infty  b_m m^{-s}\rangle_\textbf{w}=\sum_{n=n_0}^\infty  a_n\overline{b_n} w_n.\]
	For simplicity, we will take $n_0=1$ unless otherwise stated. Note that \eqref{eq:weight condition}, along with the Cauchy-Schwarz inequality, implies that $\mathcal H^\textbf{w}\subset\mathcal{D}_{\sigma_0}$.\\

	A \df{multiplier} on $\mathcal H^\textbf{w}$ is a function $\varphi$, with domain containing the half-plane $\Omega_{\sigma_0}$, with the property that $\varphi f\in\mathcal H^\textbf{w}$ for each $f\in\mathcal H^\textbf{w}$. Thus, $\varphi$	
	determines a mapping $M_\varphi:\mathcal H^\textbf{w}\longrightarrow\mathcal H^\textbf{w}$ defined by $M_\varphi f=\varphi f$ and which, by a standard application of the closed graph theorem (see Lemma \ref{lem:M is bounded} below), is a 	
	bounded operator. Further, (see Lemma \ref{prop:multiplier in H^w} below), $\varphi$ is representable by a Dirichlet series in $\Omega_{\sigma_0}$. Let $\mathfrak{M}$ denote the algebra of multipliers of $\mathcal H^{\textbf{w}}$ viewed as 
	subalgebra of the algebra of bounded operators on $\mathcal H^{\textbf{w}}$.  \\

	In \cite{HLS}, Hedenmalm, Lindqvist and Seip (HLS) classified the multipliers on $\mathcal H^\textbf{w}$ in the case that $w_n\equiv 1$ by showing that the multipliers on $\mathcal H^\textbf{w}$ were precisely the bounded, analytic functions 
	representable by a Dirichlet series in some right half-plane. More precisely,
	
	\begin{theorem}[\cite{HLS}]
		\label{HLS's theorem}
			If $w_n\equiv 1$ for each $n$, then 
				\begin{equation}
					\label{eq:isometric isomprhism}
						\mathfrak{M}\equiv H^\infty(\Omega_0)\cap\mathcal D,
				\end{equation}
			where $H^\infty(\Omega_0)$ is the space of bounded, analytic functions on $\Omega_0$ and $\mathcal D$ is the space of functions representable by a Dirichlet series in some right half-plane.
	\end{theorem}

	In \cite{M}, McCarthy extended the results of HLS to weights given by a measure.

	\begin{theorem}[\cite{M}]
		\label{McCarthy's theorem}
			If $\mu$ is a positive Radon measure with $\{0\}$ in its support and if $w_n$ is defined by 
				\begin{equation}
					\label{eq:McCarthy's weights}
						w_n=\int_0^\infty  n^{-2\sigma}d\mu(\sigma),
				\end{equation}
			where $n_0$ is the smallest natural number for which this integral is finite, then
				\begin{equation}
					\label{eq:isometric isomprhism}
						\mathfrak{M}\equiv H^\infty(\Omega_0)\cap\mathcal D.
				\end{equation}
	\end{theorem}

	In both of the above theorems, $\mathfrak{M}\equiv H^\infty(\Omega_0)\cap\mathcal D$ means $\mathfrak{M}$ is isometrically isomorphic to $H^\infty(\Omega_0)\cap\mathcal D$.\\
	
	In this article, estimates on the norms of multipliers on a class of weighted Hilbert spaces of Dirichlet series are established in the cases that the weights satisfy some basic conditions. The weights  here are complementary to those  generated by measures 	
	as considered in \cite{M}, overlapping only in the most trivial cases. Examples of weights meeting these conditions include the reciprocals of the divisor function, the sum of the divisors function and Euler's totient function. Indeed, in each of these cases 
	the norm of a multiplier is identified as the supremum norm of the symbol $\varphi$ over a right half plane.\\

	The Cauchy-Schwarz inequality tells us that norm convergence in $\mathcal H^\textbf{w}$ implies pointwise convergence in $\Omega_{\sigma_0}$ and thus, each point $u\in\Omega_{\sigma_0}$ determines a bounded point evaluation on $\mathcal 
	H^\textbf{w}$. The Riesz representation theorem then guarantees the existence of some function $k_u$ in $\mathcal H^\textbf{w}$ such that 
		\[f(u)=\langle f,k_u\rangle_\textbf{w}\]
	for each $f\in\mathcal H^\textbf{w}$. The function $k:\Omega_{\sigma_0}\times\Omega_{\sigma_0}\longrightarrow\mathbb C$ defined by $k(u,z)=k_z(u)$ is called the kernel of $\mathcal H^\textbf{w}$. Since the set 
	$\left\{w_n^{1/2}n^{-s}:n=1,2,...\right\}$ forms an orthonormal basis for $\mathcal H^\textbf{w}$, we have 
		\[k_z(u)=\langle k_z,k_u\rangle_\textbf{w}=\sum_{n=1}^\infty \langle k_z,w_n^{1/2}n^{-s}\rangle_\textbf{w}\langle w_n^{1/2}n^{-s},k_u\rangle_\textbf{w}.\]
	Working out this right-most sum yields
		\begin{equation}
			\label{eq:kernel}
				 k (u,z)=\overline{k_u(z)}=\sum_{n=1}^\infty  w_n^{-1}n^{-u-\overline{z}}.
		\end{equation}

	A sequence of weights $\{w_n\}_{n=1}^\infty $ is \df{multiplicative} if $w_{mn} = w_{m}w_n$ for each relatively prime pair of natural numbers $m$ and $n$. As a special case, $\{w_n\}_{n=1}^\infty $ is 
	\df{completely multiplicative} if $w_{mn}=w_{m}w_n$ for every pair of natural numbers $m$ and $n$ -- coprime or not.  It is straighforward to see that a multiplicative sequence is determined by its values on the powers of the primes and that a 
	completely multiplicative sequence is determined by its values on 1 and the primes. Note also that $w_1=w_1w_1$, so that $w_1=0$ or $w_1=1$. In the former case, the sequence is identically 0 and nothing interesting happens.
	Accordingly, we will relegate our discussion to the latter case.\\

	Number theory abounds with examples of multiplicative functions. For example, it is easy to see that the divisor function $d(n)$, which gives the numbers of positive divisors of the natural number $n$,  is multiplicative. Indeed, for each prime $p$ and 	
	each natural number $k$, we have $d(p^{k})=k+1$. If $m=p_1^{\alpha_1}\cdots p_M^{\alpha_M}$, then it is clear that
		\[d(m)=d(p_1^{\alpha_1}\cdots p_M^{\alpha_M})=\prod_{i=1}^M(\alpha_i+1).\]
	If $n=q_1^{\beta_1}\cdots q_n^{\beta_n}$ and if there are no primes dividing both $m$ and $n$, then
		\[d(mn)=d(p_1^{\alpha_1}\cdots p_i^{\alpha_i}q_1^{\beta_1}\cdots q_n^{\beta_n})=\prod_{i=1}^M(\alpha_i+1)\prod_{j=1}^N(\beta_j+1)=d(m)d(n).\]\\

	Let $P_N$ denote the  set consisting of the first $N$ primes and denote by $\langle P_N\rangle$ the collection of {\it words} generated by the primes in $P_N$. That is,
		\label{eq:P_N definition}
			\begin{equation}
				\langle P_N\rangle=\left\{2^{\alpha_1}3^{\alpha_2}\cdots p_N^{\alpha_N}:\alpha_j\in\mathbb N\right\}.
			\end{equation}
	Let $\mathcal H^\textbf{w}_N$ denote the closure of the subpace of $\mathcal H^\textbf{w}$ spanned by the vectors $n^{-s}$ for $n\in\langle P_N\rangle$. That is,
		\label{eq:H_N definition}
			\begin{equation}
				\mathcal H^\textbf{w}_N=\left\{f(s)=\sum_{n\in\langle P_N\rangle}a_n n^{-s}:f\in\mathcal H^\textbf{w}\right\}.
			\end{equation}
	Informally, $\mathcal H^\textbf{w}_N$ is the Hilbert space of functions obtained by taking the elements of $\mathcal H^\textbf{w}$ and ``throwing out'' the terms not indexed by $\langle P_N\rangle$. Finally, let  $\pi_N$ denote the  projection of 	
	$\mathcal H^\textbf{w}$ onto $\mathcal H^\textbf{w}_N$.\\

	There are several things that can be said about $\mathcal H^\textbf{w}_N$ and $\pi_N$.\\ 

		\begin{lemma}
			\label{lem:pi_N facts}
				Suppose $\varphi$ is a multiplier on $\mathcal H^\textbf{w}$ with 
					\[\varphi(s)=\sum_{n=1}^\infty a_n n^{-s}.\]
				Then
					\begin{enumerate}[(i)]
                                                		\item $\varphi\in\mathcal H^{\textbf{w}}$; 
						\item \label{it:pi_N definition} $\pi_N\varphi$ is a multiplier on $\mathcal H^\textbf{w}_N$ with
								\[\pi_N\varphi(s)=\sum_{n\in\langle P_N\rangle}a_n n^{-s};\]
						\item \label{it:abuse} With an abuse of notation, $\pi_N M_\varphi=M_{\pi_N\varphi}|_{\mathcal H^\textbf{w}_N}$; and 
						\item 
								\begin{equation} 
									\label{eq:pi_N k definition}
										\pi_N k^\textbf{w}(u,z)=\sum_{n\in\langle P_N\rangle}w_n^{-1}n^{-u-\overline{z}}
								\end{equation}
						is the  reproducing kernel for $\mathcal H^\textbf{w}_N$. 
					\end{enumerate}	
		\end{lemma}	

	The following is the main result of this article.  Here, $|\varphi|_{\Omega_{\sigma}}$ denotes the supremum of $\varphi$ on $\Omega_{\sigma}$.
 	
	\begin{theorem}
		\label{thm:main theorem}
			Let $\{w_n\}_{n=1}^\infty $ be a sequence of positive numbers, suppose that $\varphi$ is a multiplier of $\mathcal H^\textbf{w}$ and let $0\leq\Delta\leq\Delta'$ be real numbers.
			If 
				\[\sum_{n=1}^\infty w_n^{-1}n^{-2\sigma'}<\infty\]
			whenever $\sigma'>\Delta'$ and 
				\[\sum_{n\in\langle P_N\rangle} w_n^{-1}n^{-2\sigma}<\infty\]
			for each $N$ whenever $\sigma>\Delta$, then 
				\begin{enumerate}
					\item  each $f\in\mathcal H^\textbf{w}$ converges absolutely in $\Omega_{\Delta'}$;
					\item each $f\in\mathcal H^\textbf{w}_N$ converges absolutely in $\Omega_\Delta$;
					\item $\mathcal H^\textbf{w}$ is a reproducing kernel Hilbert space (with point evaluations being continuous in $\Omega_{\Delta'}$); 
					\item the sequence $\{\pi_N \varphi\}_{N=1}^\infty$ is uniformly bounded in sup norm on $\Omega_\Delta$  by $\|M_\varphi\|$; and  
					\item $\varphi$ converges in $\Omega_\Delta$ with
							\begin{equation} \label{eq:lower bound}|\varphi |_{\Omega_\Delta}\leq\|M_\varphi\|.\end{equation}
				\end{enumerate}

			\noindent In the other direction, if 
					\begin{enumerate}
						\item $0\leq\delta$;
						\item $\{w_n\}_{n=1}^\infty $  is multiplicative; and
						\item for each prime $p$ and positive integer $k$,  $w_{p^{k-1}}p^{-2\delta}\geq w_{p^{k}}$, 
					\end{enumerate}
			then 
				\begin{equation} \label{eq:upper bound} \|M_\varphi\|\leq|\varphi|_{\Omega_\delta}.\end{equation}
	\end{theorem}

	\begin{remark}
		 We will see in Section \ref{sec:lower bound} that if $\Delta\geq 0$ and if for each $\sigma>\Delta$ there is a $C_\sigma>0$ such that 	
			\begin{equation}
				w_n^{-1} n^{-2\sigma}\leq C_{\sigma}
			\end{equation}
		for each $n$, then
			\[|\varphi |_{\Omega_\Delta}\leq\|M_\varphi\|.\]\\
		This inequality will be a consequence of the first half of Theorem \ref{thm:main theorem}.\\
	\end{remark}

	The case where $w_n\equiv 1$ was considered by HLS in \cite{HLS}. In that case, the multipliers on $\mathcal H^\textbf{w}$ are the Dirichlet series which converge in the full right half plane to a bounded analytic function. Further,
		\[\|M_\varphi\| = |\varphi|_{\Omega_0},\]
	thus isometrically isomorphically identifying the space of multipliers on $\mathcal H^\textbf{w}$ as the space of bounded, analytic functions representable by Dirichlet series in $\Omega_0$.\\

	In the setting of Theorem \ref{McCarthy's theorem}, where the weights are given by a measure on $[0,\infty)$ for which $0$ is a point of density, the inequality in Theorem \ref{thm:growth rate theorem}
	is satisfied for each $\Delta>0$.  On the other hand, McCarthy's weights are completely multiplicative only in the case that $w_n\equiv 1$, which follows from Jensen's inequality: If $\{w_n\}_{n=1}^\infty $ is completely multiplicative, then, since
       		\[w_1=\mu([0,\infty))=1,\]
	we can apply Jenson's inequality to the convex function $x^{2}$ to see that
		\[\left(\int_0^\infty  n^{-2\sigma}d\mu(\sigma)\right)^{2}=(w_n)^{2}=w_{n^2}=\int_0^\infty  (n^{-2\sigma})^2 d\mu(\sigma).\]
	Equality only occurs when either the integrand is constant a.e.-[$\mu$] or the convex function being applied to the integral is linear. It follows then, in our case, that $\mu$ must be a point mass at some point in $[0,\infty)$. 
	This shows that a completely multiplicative sequence can't come from any positive measure at all except in the case of a point mass. In Section \ref{sec:examples}, it is shown that the multiplicative 
	number theoretic weights considered earlier are not determined by a measure.\\

	In the present context, the inequality $\|M_\varphi\|\geq|\varphi|_{\Omega_{\Delta+1}}$ follows immediately from standard reproducing kernel machinery. The sharper inequality in \eqref{eq:lower bound} is thus the content of this half of 	
	the theorem and the proof given here shares similarities with the argument found in the proof of McCarthy's result \cite{M}. However, the method used to obtain this sharper inequality employs in a crucial way a slightly more involved approach to the 	
	standard kernel-eigenfunction argument. The reverse inequality \eqref{eq:upper bound}  is established by reducing it to the case $w_n\equiv 1$ by first establishing a dilation result.  This dilation approach can also be used to recover the upper bound 
	on the multiplier norm for McCarthy's weights (see Theorem \ref{McCarthy's theorem}) from the $w_n\equiv 1$ case of Theorem \ref{HLS's theorem}.\\

	In Section \ref{sec:preliminaries}, we will be exploring some of the tools necessary to obtain our main result. In Section \ref{sec:lower bound}, we will obtain the lower bounds for our multipliers. In particular, we will be examing the case when our weights 	
	satisfy the inequality from \eqref{thm:growth rate theorem}. In Section \ref{sec:upper bound} we take a more operator theoretic approach to obtaining the upper bounds for our multipliers. We will then take a look at an alternative way to prove 	
	a special case of McCarthy's result in Section \ref{sec:McCarthy}. The article concludes with Section \ref{sec:examples}, which contains the precise results and the details of the number theoretic examples introduced earlier in this introduction.

\section{Preliminaries}
\label{sec:preliminaries}
 
	In much the same way that the Cauchy integral formula determines the coefficients for the  power series expansion of a function in the Hardy space $H^2(\mathbb D)$ of the unit disc, there is an integral formula that gives us
	the coefficients of a Dirichlet series:
		\begin{theorem}[\cite{A}]
			\label{thm:coefficients}
				If $f(s)=\sum_{n=1}^\infty a_nn^{-s}$ converges absolutely along the vertical strip $\sigma_{0}+it$, then for $x>0$, we have
					\[\lim_{T\rightarrow\infty}\frac{1}{2T}\int_{-T}^{T}f(\sigma_{0}+it)x^{\sigma_{0}+it}dt=
						\begin{cases}
							a_n	&\text{if $x=n$}\\
							0	&\text{otherwise}
						\end{cases}.\]
		\end{theorem} 
	The proof of this theorem consists of a simple Fourier coefficient argument.\\
	
	From Theorem \ref{thm:coefficients}, it then follows that if two Dirichlet series converge absolutely along some vertical strip on which they agree, then they must in fact be the same Dirichlet series (i.e. have the same coefficients). 

	\begin{theorem}[\cite{A}]
		\label{thm:f converges uniformly}
			If $f(s)=\sum_{n=1}^\infty a_n n^{-s}$ converges at the point $s_0$, then $f$ converges uniformly on each compact set in $\Omega_{s_0}$.
	\end{theorem}

	\begin{corollary}		\label{cor:f is analytic}
			If $f(s)=\sum_{n=1}^\infty a_n n^{-s}$ converges at the point $s_0$, then $f$ is analytic in $\Omega_{s_0}$.
	\end{corollary}

	\begin{proof}
		This follows from Theorem \ref{thm:f converges uniformly} and from an application of Morera's theorem.
	\end{proof}

	\begin{lemma}
		\label{lem:f's are analytic}
			Let $0\leq\Delta\leq\Delta'$. If $f\in\mathcal H^\textbf{w}$ and
				\[\sum_{n=1}^\infty  w_n^{-1}n^{-2\sigma}<\infty,\]
			whenever $\sigma>\Delta'$, then $f$ is analytic in $\Omega_{\Delta'}$. Similarly, if $f\in\mathcal H^\textbf{w}_N$ and
				\[\sum_{n\in\langle P_N\rangle}w_n^{-1}n^{-2\sigma}<\infty,\]
			whenever $\sigma>\Delta$, then $f$ is analytic in $\Omega_\Delta$.
	\end{lemma}

	\begin{proof}
		This is an immediate consequence of Corollary \ref{cor:f is analytic} and the fact that each $f\in\mathcal H^\textbf{w}$ (resp. $f\in\mathcal H^\textbf{w}_N$) converges in $\Omega_{\Delta'}$ (resp. $\Omega_\Delta$) by the Cauchy-Schwarz 
		inequality.
	\end{proof}

	The details of the above application of the Cauchy-Schwarz inequality are virtually identical to a later application, so we leave them until a more appropriate time.\\

	\begin{lemma}
		\label{lem:M is bounded}
 			If $\varphi$ is a multiplier on $\mathcal H^\textbf{w}$, then $M_\varphi $ is continuous (bounded).
	\end{lemma}
	
	\begin{proof}	
		It suffices to verify the hypotheses of the closed graph theorem. Accordingly, suppose that $f_n$, $g$ and $h$ are in $\mathcal H^\textbf{w}$, that $f_n\rightarrow g$ and that $M_\varphi f_n\rightarrow h$, both in $\mathcal H^\textbf{w}$. 	
		Note that every element of $\mathcal H^\textbf{w}$ converges absolutely along the vertical strip $\sigma+it$ whenever $\sigma>\sigma_0$. We wish to show that 
		$M_\varphi g=h$. Choose $z$ with $\mathfrak{Re}(z)>\sigma_{0}$. Since norm convergence implies point-wise convergence in $\Omega_{\sigma_{0}}$ (by Cauchy-Schwarz), we have
			\[f_n(z)\rightarrow g(z)\]
		and
			\[(M_\varphi f_n)(z)\rightarrow h(z).\]
		Now,
			\begin{align*}
				|h(z)-(M_\varphi g)(z)|&\leq|h(z)-(M_\varphi f_n)(z)|+|(M_\varphi f_n)(z)-(M_\varphi g)(z)|\\
        	                      	&=|h(z)-(M_\varphi f_n)(z)|+|\varphi(z)f_n(z)-\varphi(z)g(z)|\\
				&=|h(z)-(M_\varphi f_n)(z)|+|\varphi(z)(f_n(z)-g(z))|.\\
			\end{align*}
		Letting $n\rightarrow\infty$ shows that $(M_\varphi g)(z)=h(z)$. By Theorem \ref{thm:coefficients} we see that $M_\varphi g=h$, which is what was to be shown. 
	\end{proof}

	\begin{lemma}
		\label{prop:multiplier in H^w}
			If $\varphi$ is a multiplier on $\mathcal H^\textbf{w}$, then $\varphi\in\mathcal H^\textbf{w}$.
	\end{lemma}

	This lemma is immediate in the present setting ($n_0=1$). It is also true in the more general case that $n_{0}>1$ -- as might occur when dealing with weights like McCarthy's -- but the proof is slightly more involved. For simplicity, this case will be 	
	avoided.\\

	We close this section with the following remarkable theorem of Schnee, which which will play an important role in the proof of the Theorem \ref{thm:main theorem}.
	
	\begin{theorem}[\cite{S}]
		\label{thm:schnee}
			If $f(s)=\sum_{n=1}^\infty a_nn^{-s}$
				\begin{enumerate}
					\item converges in some (possibly remote) half-plane $\Omega_{\sigma_0}$;
					\item has an analytic continuation to $\Omega_0$; and
					\item for each $\epsilon>0$, satisfies the growth condition
						\[|f(s)|=O(|s|^\epsilon)\]
					as $|s|\rightarrow\infty$ in every right half-plane contained in $\Omega_0$,
				\end{enumerate}
		 	then $\sum_{n=1}^\infty a_n n^{-s}$ in fact converges on all of $\Omega_0$.
	\end{theorem}

\section{The Lower Bound}
\label{sec:lower bound}

	In this section, we will produce a lower bound for $M_\varphi $ in the case that the weights satisfy the convergence conditions from \eqref{eq:point evaluations} and \eqref{eq:projected point evaluations}. As a consequence, we will see that if 
	$\Delta\geq 0$ and if for each $\sigma>\Delta$ there is a $C_\sigma>0$ such that 	
		\begin{equation}
			w_n^{-1} n^{-2\sigma}\leq C_{\sigma}
		\end{equation}
	for each $n$, then
		\[|\varphi |_{\Omega_\Delta}\leq\|M_\varphi\|.\]
	We now proceed to the main theorem of this section.\\

	\begin{theorem}
		\label{thm:lower bound} Let $\{w_n\}_{n=1}^\infty $ be a sequence of positive numbers, suppose that $\varphi$ is a multiplier of $\mathcal H^\textbf{w}$ and let $0\leq\Delta\leq\Delta'$ be real numbers.
			If 
				\begin{equation}
					\label{eq:point evaluations}
						\sum_{n=1}^\infty w_n^{-1}n^{-2\sigma'}<\infty
				\end{equation}
			whenever $\sigma'>\Delta'$ and 
				\begin{equation}
					\label{eq:projected point evaluations}
						\sum_{n\in\langle P_N\rangle} w_n^{-1}n^{-2\sigma}<\infty
				\end{equation}
			for each $N$ whenever $\sigma>\Delta$, then 
				\begin{enumerate}[(i)]
					\item \label{thm:l.b.1} each $f\in\mathcal H^\textbf{w}$ converges absolutely in $\Omega_{\Delta'}$;
					\item \label{thm:l.b.2} each $f\in\mathcal H^\textbf{w}_N$ converges absolutely in $\Omega_\Delta$;
					\item \label{thm:l.b.3} $\mathcal H^\textbf{w}$ is a reproducing kernel Hilbert space (with point evaluations being continuous in $\Omega_{\Delta'}$); 
					\item \label{thm:l.b.4} the sequence $\{\pi_N \varphi\}_{N=1}^\infty$ is uniformly bounded in sup norm on $\Omega_\Delta$ by $\|M_\varphi\|$; and  
					\item \label{thm:l.b.5} $\varphi$ converges in $\Omega_\Delta$ with
							\[|\varphi |_{\Omega_\Delta}\leq\|M_\varphi\|.\]
				\end{enumerate}
	\end{theorem}

	\begin{proof}
		An application of the Cauchy-Schwarz inequality shows that point evaluations are continuous in $\Omega_{\Delta'}$: If $s=\sigma+it\in\Omega_{\Delta'}$ and if $f(s)=\sum_{n=1}^\infty f_n n^{-s}$, then
			\begin{align*}
				|f(s)|&\leq\sum_{n=1}^\infty |f_n|n^{-\sigma}\\
				&=\sum_{n=1}^\infty |f_n|w_n^{1/2} w_n^{-1/2}n^{-\sigma}\\
				&\leq\sqrt{\sum_{n=1}^\infty |f_n|^{2}w_n}\sqrt{\sum_{n=1}^\infty w_n^{-1}n^{-2\sigma}}\\
				&=\|f\|_\textbf{w}\sqrt{\sum_{n=1}^\infty w_n^{-1}n^{-2\sigma}}.\\
			\end{align*}
		Thus, $\mathcal H^\textbf{w}$ is a RKHS in which each element is representable by an absolutely convergent Dirichlet series in $\Omega_{\Delta'}$.\\
	
		Turning our attention to the projected subspace 
			\[\mathcal H^\textbf{w}_N=\left\{f(s)=\sum_{n\in\langle P_N\rangle}a_n n^{-s}:f\in\mathcal H^\textbf{w}\right\},\] 
		an application of the Cauchy-Schwarz inequality to the function
			\[f(s)=\sum_{n\in\langle P_N\rangle}f_n n^{-s}\]
		shows -- as above -- that point evaluations are continuous in $\Omega_\Delta$. This proves \eqref{thm:l.b.1}, \eqref{thm:l.b.2} and \eqref{thm:l.b.3}. 
		Moreover,  since -- as  noted in Equation \eqref{eq:pi_N k definition} -- if $k(u,z)$ is the reproducing kernel for $\mathcal H^\textbf{w}$, then $\pi_N k(u,z)$ is the reproducing kernel for the projected space $\mathcal H^\textbf{w}$ and Lemma 
		\ref{lem:f's are analytic} tells us that each $\pi_N k_u$ is analytic in $\Omega_\Delta$.\\

		 From the standard kernel/eigenfunction argument, we have
			\begin{equation}
				\label{eq:phi_N bounded}
					|\pi_N\varphi|_{\Omega_{\Delta}}\leq\|\pi_N M_\varphi |_{\mathcal H^\textbf{w}_N}\|\leq\|\pi_N M_\varphi \|\leq\|M_\varphi \|,
			\end{equation}
		with $\pi_N\varphi$ converging in $\Omega_\Delta$ by item \eqref{thm:l.b.2}. By a normal families argument, the sequence $\{\pi_N\varphi\}_{n=1}^\infty$ has a subsequence $\{\pi_{N_j}\varphi\}_{j=1}^\infty $ which converges uniformly 		
		on compact sets in $\Omega_\Delta$ to some function $\psi$, analytic in $\Omega_\Delta$. Hence 
			\[|\psi |_{\Omega_\Delta}\leq\|M_\varphi\|.\]
		On the other hand, the sequence $\{\pi_{N_j}\varphi\}_{j=1}^\infty $ converges to $\varphi$ uniformly on compact subsets of $\Omega_{\Delta'}$ by the Cauchy-Schwarz inequaility since, if $\varphi(s)=\sum_{n=1}^\infty a_n n^{-s}$ and if
		$\sigma=\mathfrak{Re}(s)>\Delta'$, then
			\[|\varphi(s)-\pi_{N_j}\varphi(s)|\leq\sqrt{\sum_{n\in\mathbb N\backslash\langle P_{N_j}\rangle}|a_n|^2 w_n}\sqrt{\sum_{n\in\mathbb N\backslash\langle P_{N_j}\rangle}w_n^{-1}n^{-2\sigma}},\]
		with the right-hand side tending to 0 as $j\to\infty$. It follows that $\psi=\varphi$ on $\Omega_{\Delta'}$. Since $\psi$ is a bounded analytic continuation of $\varphi$ into $\Omega_\Delta$, Schnee's theorem (a slight variation 	
		actually) tells us that the Dirichlet series for $\varphi$ converges on $\Omega_\Delta$. Corollary \ref{cor:f is analytic} tells us that $\varphi$ is analytic in $\Omega_\Delta$, and as $\varphi=\psi$ in $\Omega_\Delta'$, we have $\varphi=\psi$ in 
		$\Omega_\Delta$. It follows that $|\varphi|_{\Omega_\Delta}\leq\|M_\varphi \|$.
	\end{proof}

	Some facts which arose in the above proof, will be collected in the following theorem.

	\begin{theorem}	
		\label{thm:normal familes} 
         			Let $\varphi$ be a multiplier on the space $\mathcal H^\textbf{w}.$ If the sequence $\{\pi_N\varphi\}_{N=1}^\infty$ is uniformly bounded by (say) $B$ in supnorm in the half-plane $\Omega_\Delta$, then  
				\begin{enumerate}[(i)]
					\item \label{thm:n.f.1} there is some subsequence $\{\pi_{N_j}\varphi\}_{j=1}^\infty$ converging pointwise to $\varphi$ in $\Omega_\Delta$;
					\item \label{thm:n.f.2} $\varphi$ converges in $\Omega_\Delta$; and
         					\item \label{thm:n.f.3} $|\varphi|_{\Omega_\Delta}\leq B$.
				\end{enumerate}
	\end{theorem}
 
	We now move to a more interesting condition on our weights, which will allow us to obtain the convergence conditions given in Theorem \ref{thm:main theorem} and which -- as we will see --  be more useful in producing bounds for our multipliers when our 		weights are of a more number theoretic variety.  But first, a lemma.\\

	\begin{lemma}
		\label{lem:converge in Omega_0}
			For each positive integer $N$, the series given by 
				\[\sum_{n\in\langle P_N\rangle}n^{-s}\]
			converges in $\Omega_0$.
	\end{lemma}

	\begin{proof}
		We have
			\[\sum_{n\in\langle P_N\rangle}n^{-s}=\prod_{p\in P_N}\frac{1}{1-p^{-s}}\]
		far enough to the right. Since $\sum_{n\in\langle P_N\rangle}n^{-s}$ has an analytic continuation to a bounded function in $\Omega_{\epsilon}$  
           	for each $\epsilon>0,$  Schnee's theorem, Theorem \ref{thm:schnee}, tells us that $\sum_{n\in\langle P_N\rangle} n^{-s}$ in fact converges in all of $\Omega_\epsilon$. Since $\epsilon$ was arbitrary, it follows that 
		$\sum_{n\in\langle P_N\rangle}n^{-s}$ converges in $\Omega_0$.
	\end{proof}

	\begin{remark}  
		 Schnee is not actually needed in this proof, but as it needed in the proof of Theorem \ref{thm:lower bound}, it seems sensible to use it here as well.\\
	\end{remark}

	\begin{theorem}
		\label{thm:growth rate theorem}
	   		Let $\Delta\geq 0$ and a sequence $\{w_n\}_{n=1}^\infty$ of positive numbers be given. If for each $\sigma>\Delta$, there exists a $C_\sigma>0$ such that
				\[w_n^{-1} n^{-2\sigma}\leq C_{\sigma}\]
			for each $n$, and if $\varphi$ is a multiplier of $\mathcal H^{\textbf{w}}$, then the inequalities in Theorem \ref{thm:lower bound} are satisfied by $\Delta+\frac{1}{2}$ and $\Delta$ respectively and
				\[|\varphi |_{\Omega_\Delta}\leq\|M_\varphi\|.\]\\
	\end{theorem}
	 
	\begin{proof}
		Let $\epsilon>0$ and let $\sigma>\Delta+\frac{1}{2}+\epsilon$. The Cauchy-Schwarz inequality gives
			\begin{align*}
				\sum_{n=1}^\infty w_n^{-1}n^{-2\sigma}&\leq\sum_{n=1}^\infty w_n^{-1}n^{-2(\Delta+\epsilon/2)}n^{-1-\epsilon}\\
				&\leq C_{\Delta+\epsilon/2}\sum_{n=1}^\infty n^{-1-\epsilon}\\
				&<\infty,
			\end{align*}
		and we see that the inequality in \eqref{eq:point evaluations} of Theorem \ref{thm:lower bound} is satisfied by$\Delta'=\Delta+\frac{1}{2}$.\\
	
		We will now show that condition \eqref{eq:projected point evaluations} of Theorem \ref{thm:lower bound} is satisfied by $\Delta$. To do this, we will make use of a truncated version of the famed Euler product
			\[\sum_{n=1}^\infty n^{-s}=\prod_{p\text{ prime }}\frac{1}{1-p^{-s}},\]
		given by
			\[\sum_{n\in\langle P_N\rangle}n^{-s}=\prod_{p\in P_N}\frac{1}{1-p^{-s}}.\]
	
		Let $\epsilon>0$ and let $\sigma>\Delta+\epsilon$. Observe that 
			\[\sum_{n\in\langle P_N\rangle}w_n^{-1}n^{-2\sigma}\leq\sum_{n\in\langle P_N\rangle}w_n^{-1}n^{-2(\Delta+\epsilon/2)}n^{-\epsilon}\leq C_{\Delta+\epsilon/2}\sum_{n\in\langle P_n\rangle}n^{-\epsilon}.\]
		The right-most sum converges by Lemma  \ref{lem:converge in Omega_0}, so that condition \eqref{eq:projected point evaluations} of Theorem \ref{thm:lower bound} is satisfied by $\Delta$. It now follows from Theorem \ref{thm:lower bound} 
		that 
			\[|\varphi |_{\Omega_\Delta}\leq\|M_\varphi\|.\]
	\end{proof}

\section{The Upper Bound}
\label{sec:upper bound}

	The second half of Theorem \ref{thm:main theorem} is established in this section.

	\begin{theorem}
		\label{thm:upper bound} 
			Let $\{w_n\}_{n=1}^\infty $ be a sequence of positive numbers, let $\varphi$ be a multiplier of $\mathcal H^\textbf{w}$ and let $\delta\geq 0$. If 
				\begin{enumerate}[(i)]
					\item \label{eq:u.b. condition 1} $\{w_n\}_{n=1}^\infty $  is multiplicative; and 
	 				\item \label{eq:u.b. condition 2} for each prime $p$ and positive integer $k$, we have $w_{p^k}\leq p^{-2\delta}w_{p^{k-1}}$, 
				\end{enumerate}
			then 
				\[\|M_\varphi\|\leq|\varphi|_{\Omega_\delta}.\]
	\end{theorem}
 
	By replacing $w_n$ by $w_n n^{-2\delta}$, it suffices to establish the result for $\delta=0$. In this case, \eqref{eq:u.b. condition 2} can be restated as saying that the sequence $\{w_{p^k}\}_{k=0}^\infty$ is decreasing for each prime $p$. The case 	
	where $w_n\equiv 1$ appears in \cite{HLS,H,M} and is used here to establish this more general result.

	\begin{proof} 
		Let $\mathcal H^0$ and $k^0$ denote the space and kernel corresponding to the weights $w_n \equiv 1$. Note that for $u$ and $z$ in $\Omega_1$, the kernel $k^0(u,z)$ converges absolutely and is given by 
			\label{eq:kernel product}
				\begin{equation}
					k^0(u,z)=\sum_{n=1}^\infty n^{-\tau}=\prod_{p\text{ prime }}\frac{1}{1-p^{-\tau}},
				\end{equation}
		where $\tau=u+\overline{z}$. For the general sequence of weights and for $u$ and $z$ in $\Omega_{\sigma_0}$ (with $\sigma_0$ as in \eqref{eq:weight condition}), the kernel $k^\textbf{w}(u,z)$ converges absolutely, and has its own product 
		representation given by
			\begin{equation}
				\label{eq:general product}
					\prod_{p\text{ prime }}\sum_{j=0}^\infty w_{p^j}^{-1}p^{-j\tau}.
			\end{equation}
			
		Let $m=\max\{1,\sigma_0\}$, let $\tau\in\Omega_m$ and let
			\label{equ:kernel ratio}
				\begin{equation}
					K(u,z)=\frac{k^\textbf{w}(u,z)}{k^0(u,z)},
   				\end{equation}
   		which is well-defined since $k^0(u,z)$ has no zeros $\Omega_1\times\Omega_1$. The products for these kernels converge giving,
			\begin{align*}
				K(u,z)&=\prod_{p\text{ prime }}(1-p^{-\tau})\sum_{j=0}^\infty w_{p^{j}}^{-1} p^{-j\tau}\\
				&=\prod_{p\text{ prime }}\left(1+\sum_{j=1}^\infty(w_{p^j}^{-1}-w_{p^{j-1}}^{-1})p^{-j\tau}\right).
			\end{align*}
		Because the weights are assumed to be decreasing by condition \eqref{eq:u.b. condition 2} of Theorem \ref{thm:upper bound}, we have
			\[w_{p^j}^{-1}-w_{p^{j-1}}^{-1}\geq 0,\]
		and it follows that $K(u,z)$ is positive semidefinite on $\Omega_m\times\Omega_m$.\\
	
		From the theory of reproducing kernels, the fact that $K$ is positive semidefinite implies the existence an auxiliary Hilbert space $\mathcal H^Q$ and a function $Q:\Omega_m\to\mathcal H^Q$ such that
			\[K(u,z)=\frac{k^\textbf{w}(u,z)}{k^0(u,z)}=Q(z)^*Q(u).\]
		Multiplying through and rewriting, we have
			\begin{equation}
				\label{eq:pre V}
					\langle k^\textbf{w}_u, k^\textbf{w}_z\rangle_\textbf{w}=\langle k^0_u, k^0_z\rangle_0\langle Q(u),Q(z)\rangle_Q=\langle k^0_u\otimes Q(u), k^0_z\otimes Q(z)\rangle_\otimes.
			\end{equation}
		Define an operator $V$ from the set $\mathcal S=\{k^\textbf{w}_u:u\in\Omega_m\}$ to the set $\{k^0_u\otimes Q(z)\}$ by 
			\label{eq:isometry}
				\begin{equation}
					Vk^\textbf{w}_u=k^0_u\otimes Q(u).
				\end{equation}
		We can extend this $V$ by linearity to a map -- still denoted by $V$ -- to $\text{span }\mathcal S$. This $V$, so defined, is an isometry on $\text{span }\mathcal S$. Since $\text{span }\mathcal S$ is dense in $\mathcal H^\textbf{w}$, $V$ thus 
		extends continuously to an isometry, still denoted by $V$ on all of $\mathcal H^\textbf{w}$ into $\mathcal H^0\otimes H^Q$. \\
	
		For the multiplier $\varphi$ on $\mathcal H^\textbf{w}$ and for $u\in\Omega_m$ it is well known that  
			\[M_\varphi^* k^\textbf{w}_u=\overline{\varphi(u)}k^\textbf{w}_u.\]
		If  $\varphi$ is unbounded in $\Omega_0$, then there's nothing to do since then $|\varphi|_{\Omega_0}=\infty$. Otherwise, $\varphi$ is a bounded Dirichlet series on $\Omega_0$ and it follows from Theorem \ref{HLS's theorem}
		that $\varphi$ is a multiplier of $\mathcal H^0$ and $\|M_\varphi\|_0\leq|\varphi|_{\Omega_0}$.\\
	
		Let $M_{\varphi,\textbf{w}}$ denote multiplication by $\varphi$ in $\mathcal H^\textbf{w}$ and let $M_{\varphi,0}$ denote multiplication by $\varphi$ in $\mathcal H^0$. For $k_u^0\otimes Q(z)$ in $\mathcal H^0\otimes\mathcal H^Q$, we have  
			\label{eq:tensormap}
				\begin{equation}
					(M_{\varphi,0}^*\otimes I)(k^0_u\otimes Q(z))=\overline{\varphi(u)}k^0_u\otimes Q(z),
				\end{equation}
		and it follows that
			\label{eq:intertwining}
				\begin{equation}
					VM_{\varphi,\textbf{w}}^*=(M_{\varphi,0}^*\otimes I)V.
				\end{equation}
		It is then easy to see that the following diagram commutes:
			\[\begin{CD}
				\mathcal H^\textbf{w} @>V>>\mathcal H^0\otimes\mathcal H^Q\\
				@VM_{\varphi,\textbf{w}}^* VV @VVM_{\varphi,0}^*\otimes I V\\
				\mathcal H^\textbf{w} @>V>>\mathcal H^0\otimes\mathcal H^Q.\\
			\end{CD}\]
		In particular, we see that
			\[\|M_{\varphi,\textbf{w}}\|\leq\|V^{*}\|\;\|M_{\varphi,0}\|\;\|V\|\leq\|M_{\varphi,0}\|\leq|\varphi|_{\Omega_0},\]
		which is what was to be shown. 
	\end{proof} 

\subsection{The Case When $\delta=\Delta$}
\label{ssec:isometry}

	Now suppose that $w_n$ defines a sequence of positive numbers satisfying the conditions of Theorems \ref{thm:lower bound} and \ref{thm:upper bound} with $0\leq\delta\leq\Delta\leq\Delta'$, restated here for convenience:
		\begin{enumerate}
			\item $\sum_{n=1}^\infty w_n^{-1}n^{-2\sigma'}<\infty$ whenever $\sigma'>\Delta'$;
			\item $\sum_{n\in\langle P_N\rangle} w_n^{-1}n^{-2\sigma}<\infty$ for each $N$ whenever $\sigma>\Delta$;
			\item $\{w_n\}_{n=1}^\infty $  is multiplicative; and
	 		\item for each prime $p$, we have $w_{p^k}\leq p^{-2\delta}w_{p^{k-1}}$.
		\end{enumerate}

	We've seen that $|\varphi |_{\Omega_\Delta}\leq\|M_\varphi\|$ and $|\varphi |_{\Omega_\delta}\geq\|M_\varphi\|$. In particular, if $\delta=\Delta$, then the map $M_\varphi\mapsto\varphi$ is an isometry from $\mathfrak{M}$, the space of multipliers 	
	on $\mathcal H^\textbf{w}$ with operator norm, into $H^\infty(\Omega_\delta)\cap\mathcal D$, the space of functions, bounded and holomorphic in $\Omega_\delta$ which are representable by Dirichlet series in some right half-plane -- this space being 		equipped with the sup norm. The content of this next theorem is to show that this map is in fact surjective -- that each $\varphi\in H^\infty(\Omega_\delta)\cap\mathcal D$ gives rise to a multiplier on $\mathcal H^\textbf{w}$.\\

	\begin{theorem}
		\label{thm:surjection}
			If $\varphi\in H^\infty(\Omega_\delta)\cap\mathcal D$, then $\varphi$ is a multiplier on $\mathcal H^\textbf{w}$.
	\end{theorem}

	\begin{proof}
		There is no loss in assuming that $\delta=0$, as before. Since $\varphi\in H^\infty(\Omega_0)\cap\mathcal D$, Theorem \ref{HLS's theorem}, tells us that $\varphi$ is a multiplier on $\mathcal H^0$. As before, denote the 
		corresponding multiplication operator on $\mathcal H^0$ by $M_{\varphi,0}$. Note that $\mathcal H^0\subset\mathcal H^\textbf{w}$ since $w_n\leq 1$ for each $n$ -- which follows from the facts that our weights are multiplicative and $w_1=1$. 
		Now, if $f\in\mathcal H^\textbf{w}$ is a Dirichlet series, then with $V$ defined as in the proof of Theorem \ref{thm:upper bound}, we have
			\begin{align*}
				\langle V^*(M_{\varphi,0}\otimes I)Vf,k_\lambda^\textbf{w}\rangle_\textbf{w}&=\langle f,V^*(M_{\varphi,0}^*\otimes I)Vk_\lambda^\textbf{w}\rangle_\textbf{w}\\
				&=\langle f,V^*(\overline{\varphi(\lambda)}k_\lambda^0\otimes Q(\lambda))\rangle_\textbf{w}\\
				&=\varphi(\lambda)\langle f,V^*(k_\lambda^0\otimes Q(\lambda))\rangle_\textbf{w}\\
				&=\varphi(\lambda)\langle f,V^*Vk_\lambda^\textbf{w}\rangle_\textbf{w}\\
				&=\varphi(\lambda)\langle Vf,Vk_\lambda^\textbf{w}\rangle_\textbf{w}\\
				&=\varphi(\lambda)\langle f,k_\lambda^\textbf{w}\rangle_\textbf{w}\\
				&=\langle\varphi f,k_\lambda^\textbf{w}\rangle_\textbf{w}.\\
			\end{align*}
		Thus,
			\label{eq:phi a multiplier}
				\begin{equation}
					\langle V^*(M_{\varphi,0}\otimes I)Vf,k_\lambda^\textbf{w}\rangle_\textbf{w}=\langle\varphi f,k_\lambda^\textbf{w}\rangle_\textbf{w}
				\end{equation}
		for each $f\in\mathcal H^\textbf{w}$. As the span of the kernel functions $k_\lambda$ is dense in $\mathcal H^\textbf{w}$, we see that
			\[\langle V^*(M_{\varphi,0}\otimes I)Vf,g\rangle_\textbf{w}=\langle\varphi f,g\rangle_\textbf{w}\]
		for each $g\in\mathcal H^\textbf{w}$ and it follows that $V^*(M_{\varphi,0}\otimes I)V$ is given by multiplication by $\varphi$. 
	\end{proof}

	\begin{corollary}
		\label{corollary}
			If $0\leq\delta=\Delta\leq\Delta'$ and
				\begin{enumerate}
					\item $\sum_{n=1}^\infty w_n^{-1}n^{-2\sigma'}<\infty$ whenever $\sigma'>\Delta'$;
					\item $\sum_{n\in\langle P_N\rangle} w_n^{-1}n^{-2\sigma}<\infty$ for each $N$ whenever $\sigma>\Delta$;
					\item $\{w_n\}_{n=1}^\infty $  is multiplicative; and
			 		\item for each prime $p$, we have $w_{p^k}\leq p^{-2\delta}w_{p^{k-1}}$.
				\end{enumerate}
			then
				\[\mathfrak{M}\equiv H^\infty(\Omega_\delta)\cap\mathcal D.\]
	\end{corollary}

	\begin{proof}
		The map $M_\varphi\mapsto\varphi$ is an isometry by Theorems \ref{thm:lower bound} and \ref{thm:upper bound} and is surjective by Theorem \ref{thm:surjection}.
	\end{proof}

	\begin{remark} 
		\label{rm:3 conditions}
			Note that the first two conditions in Corollary \ref{corollary} are satisfied by $\Delta+\frac{1}{2}$ and $\Delta$ when the growth condition in Theorem \ref{thm:growth rate theorem} is met. Corollary \ref{corollary} may then be restated as 				follows:
				If $0\leq\delta=\Delta$ and
					\begin{enumerate}
						\item  if for each $\sigma>\Delta$ there is a $C_\sigma>0$ such that 	
								\begin{equation*}
									w_n^{-1} n^{-2\sigma}\leq C_{\sigma}
								\end{equation*}
						for each $n$;
						\item $\{w_n\}_{n=1}^\infty $  is multiplicative; and
				 		\item for each prime $p$, we have $w_{p^k}\leq p^{-2\delta}w_{p^{k-1}}$,
					\end{enumerate}
				then
					\[\mathfrak{M}\equiv H^\infty(\Omega_\delta)\cap\mathcal D.\]
	\end{remark}

\section{Weights Given by Measures}
\label{sec:McCarthy}
	
	Given $\mu,$ a Borel probability measure on $[0,\infty)$ with $0$ in its support, define
		\begin{equation}
			\label{eq:w_n mu}
				w_n=\int_0^\infty n^{-2\sigma}d\mu(\sigma).
		\end{equation}
	For ease of exposition, we have restricted our attention the class of probability measures -- and corresponding weight sequences -- which is less general than that considered by McCarthy in \cite{M}.

	\begin{theorem}[McCarthy \cite{M}]
		The multiplier algebra of $\mathcal H^\textbf{w}$ is isometrically isomorphic to $H^\infty (\Omega_{0})\cap\mathcal D$, where the norm on $H^\infty (\Omega_{0})\cap\mathcal D$ is the 
		supremum of the absolute value on $\Omega_{0}$.
	\end{theorem}

	\begin{proof}
		The hypothesis that $0$ is the left hand endpoint of the support of $\mu$ implies that for every $\sigma>0$ there is a $C_\sigma>0$ such that 
			\[w_n^{-1}n^{-2\sigma}\leq C_\sigma.\]
		In particular, if $\varphi$ is a multiplier of $\mathcal H^\textbf{w}$, then by using Theorem \ref{thm:growth rate theorem} in conjunction with \eqref{thm:l.b.5} of Theorem \ref{thm:lower bound}, we see that the Dirichlet series for $\varphi$ 
		converges on $\Omega_0$ and that
			\[|\varphi|_{\Omega_0}\leq\|M_\varphi \|.\]
	
		Given $\delta\geq 0$, let $\mathcal H^\delta$ denote the Hilbert space of Dirichlet series corresponding to the weight sequence $\{n^{-2\delta}\}_{n=1}^\infty $, let $\|f\|_\delta$ and $\langle f,g\rangle_\delta$
		denote the norm and inner product in $\mathcal H^\delta$ respectively and let $\mathcal F$ denote the set of all continuous, bounded functions from $[0,\infty)$ into $\mathcal H^0$. Observe that $\mathcal H^0$ includes
	    	(contractively) into $\mathcal H^\delta$ as well as into $\mathcal H^{\mathbf{w}}$. In particular, for $F\in\mathcal F$,
			\[\int_0^\infty\|F(\delta)\|^2_\delta \, d\mu(\delta)\leq\int_0^\infty\|F(\delta)\|_0^2 \, d\mu(\delta)\leq C^2,\]
		where $C$ is a bound for $F$ in the sense that $\|F(\delta)\|\leq C$ for each $\delta\geq 0$.\\

		On $\mathcal{F}$, consider the inner product given by
			\label{eq:new product}
				\begin{equation}
					\langle F,G\rangle=\int_0^\infty\langle F(\delta),G(\delta)\rangle_\delta d\mu(\delta)
				\end{equation}
		and let $\mathcal{F}^2$ denote the resulting Hilbert space.  Define $W$ on the dense subset $\mathcal H^0$ of $\mathcal H^\textbf{w}$ by inclusion as constant functions:
			\label{eq:W definition}
				\begin{equation}
					(Wf)(\delta)=f.
				\end{equation}
		Observe that, for finite Dirichlet series $f(s)=\sum_{n=1}^m f_n n^{-s}$ and $g=\sum_{n=1}^m g_n n^{-s}$ in $\mathcal H^0$, we have
			\[\begin{split}
				\langle Wf, Wg\rangle=&\int_0^\infty\langle f,g\rangle_\delta d\mu(\delta) \\
				&=\sum_{n=1}^m f_n\overline{g_n}\int_0^\infty n^{-2\delta}d\mu(\delta) \\
				&=\sum_{n=1}^m f_n\overline{g_n}w_n.\\
				&=\langle f,g\rangle_\textbf{w}.
			\end{split}\]
		It follows that $W$ extends to an isometry (still denoted by $W$) from $\mathcal H^\textbf{w}$ into $\mathcal F^2$.

		As already noted, the Dirichlet series for $\varphi$ converges on all of $\Omega_0$. It follows that if $f$ is a finite Dirichlet series (finitely many nonzero terms), then $\varphi f\in\mathcal H^0$. Hence,
			\[(W\varphi f)(\delta)=\varphi f.\]
		On the other hand, given $F\in\mathcal F$, the function $JF(\delta)=\varphi F(\delta)$ is in $\mathcal F$ since $\varphi$ is a multiplier of $\mathcal H^0$. The HLS result with $\delta$ in place of $0$ gives, 
			\[\begin{split}
				\|JF\|^2&=\int_0^\infty\|\varphi F(\delta)\|_\delta^2 d\mu(\delta)\\	
				&\leq\int_0^\infty|\varphi|_{\Omega_\delta}^2\|F(\delta)\|_\delta^2 d\mu(\delta)\\
				&\leq|\varphi|_{\Omega_0}^2\int_0^\infty\|F(\delta)\|^2_\delta d\mu(\delta)\\
				&=|\varphi|_{\Omega_0}^2\|F\|^2.
			\end{split}\]
		Thus $J$ extends to a bounded operator on $\mathcal F^2$ with $\|J\|\leq|\varphi|_{\Omega_0}$. To complete the proof, observe that 
			\[WM_\varphi =JW,\]
		from which we can see that
			\[\|M_\varphi \|\leq\|J\|.\]
	\end{proof}

\section{Examples}
\label{sec:examples}

 	We close this article by looking at some examples of weight sequences defined by certain well-known arithmetic functions. We've already seen that -- with the exception of weights coming from point-masses --
	completely multiplicative sequences can't arise from a measure like McCarthy's. The Jensen's inequality argument used in the introduction doesn't carry over to the case when the weights are multiplicative, but not completely so.\\

	\begin{definition}
		\label{divisor def}The divisor function -- denoted by $d(n)$ -- gives the number of positive divisors of the natural number $n$. If $n=p_{1}^{\alpha_{1}}\cdots p_{m}^{\alpha_{m}}$, then as we saw in the introduction, 
			\[d(n)=\prod_{j=1}^{m}(\alpha_{j}+1).\]
		\label{sum or divisors def}The sum of divisors function -- denoted by $\sigma(n)$ -- gives the sum of the divisors of the natural number $n$. If $n=p_{1}^{\alpha_{1}}\cdots p_{m}^{\alpha_{m}}$, then
			\[\sigma(n)=\prod_{j=1}^{m}\left(\frac{p_{j}^{\alpha_{j}+1}-1}{p_{j}-1}\right).\]
		\label{Euler totient def}The Euler totient function -- denoted by $\phi(n)$ -- tells how many positive integers less than $n$ are coprime to $n$. If $n=p_{1}^{\alpha_{1}}\cdots p_{m}^{\alpha_{m}}$, then
			\[\phi(n)=n\prod_{j=1}^{m}\left(1-\frac{1}{p_{j}}\right).\]
	\end{definition}

	Each of these functions is multiplicative; none are completely so. In this section, we will be working with sequences given by the reciprocals of these functions. Any such sequence is automatically multiplicative.  \\

	\begin{theorem}
		\label{thm:other isometries}
		If $w_n$ is given by the reciprocal of $d(n)$, $\sigma(n)$ or $\phi(n)$, then there is no positive measure $\mu$ such that
			\begin{equation}
				w_n=\int_{0}^\infty n^{-2\sigma}d\mu(\sigma).
			\end{equation}
		On the other hand, in each case there is a $\delta$ such that $\mathfrak{M}$ is isometrically isomorphic to $H^\infty(\Omega_\delta)\cap\mathcal D$.
	\end{theorem}

	Before moving into the proof of this theorem, some preliminaries are in order.

	\begin{lemma} For each natural number $n$, 
		\label{lem:inequality}
			\[\frac{6}{\pi^2}<\frac{\sigma(n)\phi(n)}{n^2}<1.\]
	\end{lemma}
 
	\noindent\textbf{Growth Rates:} 
		The following growth rates, the first of which is known as Gr\"{o}nwall's Theorem, hold.\\
		
			\noindent 1. Gr\"{o}nwall's Theorem:
				\begin{equation}
					\label{eq:Gronwall}
						\limsup_n\frac{\sigma(n)}{n\ln\ln n}=e^{\gamma}
				\end{equation}
			where $\gamma$ is the Euler-Mascheroni constant.\\
			2. For each $\epsilon>0$, 
				\begin{equation}
					\label{eq:sigma}
						\sigma(n)=O(n^{1+\epsilon}).
				\end{equation}
			3.
				\begin{equation}
					\label{eq:phi 1} 
						\limsup_n\frac{\phi(n)}{n}=1; 
				\end{equation}
			4. 
				\begin{equation}
					\label{eq:phi 2}
						\liminf_n\frac{\phi(n)}{n}=0.
				\end{equation}
	
	\noindent We are now in a position to prove Theorem \ref{thm:other isometries}. 

	\begin{proof}

		Suppose first that $w_n=\frac{1}{d(n)}$. Assume by way of contradiction that there is some positive measure $\mu$ such that
			\[w_n=\frac{1}{d(n)}=\int_{0}^\infty n^{-2\sigma}d\mu(\sigma).\]
		Let $\{p_1<p_2<p_3<\cdots\}$ be the primes for which this integral is finite. Then
			\[\frac{1}{2}=\frac{1}{d(p_j)}=\int_0^\infty p_j^{-2\sigma}d\mu(\sigma).\]
		Letting $j\to\infty$ and using the dominated convergence theorem, we see that $\frac{1}{2}=\mu(\{0\})$. Now, apply the same argument with the sequence $\{p_1^2<p_2^2<p_3^2<\cdots\}$, along with the fact that
		$\frac{1}{d(p_j^2)}=\frac{1}{3}$ to get $\mu(\{0\})=\frac{1}{3}$ -- a contradiction.\\
	
		Suppose now that  $w_n=\frac{1}{\sigma(n)}$ and again, assume by way of contradiction that $w_n$ is given by a measure as in Equation \eqref{eq:McCarthy's weights}. Gr\"{o}nwall's theorem implies that there is some sequence 
		$\{n_j\}_{j=1}^\infty $ such that
			\[\lim_{j\to\infty}\frac{n_j\ln\ln n_j}{\sigma(n_j)}=e^{-\gamma}.\]
		Hence
			\begin{align*}
				e^{-\gamma}&=\lim_{j\to\infty}\frac{n_j\ln\ln n_j}{\sigma(n_j)}\\
				&=\lim_{j\to\infty}\int_0^\infty n_j^{1-2\sigma}\ln\ln n_j d\mu(\sigma)\\
				&\geq\lim_{j\to\infty}\int_{[0,1/2]}n_j^{1-2\sigma}\ln\ln n_jd\mu(\sigma)\\
				&\geq\lim_{j\to\infty}\mu([0,1/2])\ln\ln n_j.
			\end{align*}
		It follows that $\mu([0,1/2])=0$.\\

		Since
			\[\frac{6}{\pi^2}<\frac{\sigma(n)\phi(n)}{n^2}<1\]
		(by Lemma \ref{lem:inequality}), \eqref{eq:phi 1} implies the existence of a sequence $\{n_i\}_{i=1}^\infty$ such that
			\[\frac{6}{\pi^2}\leq\frac{\sigma(n_i)}{n_i}\leq 2,\]
		for each $i$, whence 
			\[2\leq\limsup_i\frac{n_i}{\sigma(n_i)}.\]
		Thus
			\begin{align*}
				2&\leq\limsup_i\frac{n_i}{\sigma(n_i)}\\
				&=\limsup_i\int_0^\infty n_i^{1-2\sigma}d\mu(\sigma)\\	
				&=\limsup_i\left(\int_{[0,1/2]}n_i^{1-2\sigma}d\mu(\sigma)+\int_{(1/2,\infty)}n_i^{1-2\sigma}d\mu(\sigma)\right)\\
				&=\limsup_i\left(\int_{(1/2,\infty)}n_i^{1-2\sigma}d\mu(\sigma)\right)\\
				&=0,
			\end{align*}
		with the second-to-last step coming from the fact that $\mu([0,1/2])=0$ and the last step coming from Fatou's lemma. This contradiction proves that $w_n$ so defined can't come from a measure.\\

		Finally, let's $w_n=\frac{1}{\phi(n)}$ and again, suppose by way of contradiction, that there is a $\mu$ satisfying Equation \eqref{eq:McCarthy's weights}. By \eqref{eq:phi 1}, there is a sequence $\{n_i\}_{i=1}^\infty $ such that
			\[\lim_{i\to\infty}\frac{n_i}{\phi(n_i)}=1.\]
		It follows that
			\begin{align*}
				1&=\lim_{i\to\infty}\frac{n_i}{\phi(n_i)}\\
				&=\lim_{i\to\infty}\int_0^\infty n_i^{1-2\sigma}d\mu(\sigma)\\
				&=\lim_{i\to\infty}\left(\int_{[0,1/2)}n_i^{1-2\sigma}d\mu(\sigma)+\mu(\{1/2\})+\int_{(1/2,\infty)}n_i^{1-2\sigma}d\mu(\sigma)\right).
			\end{align*}
		Now, the first integral in the parentheses must be 0 and the second integral in the parentheses must tend to 0 by the dominated convergence theorem, so that $\mu(\{1/2\})=1$. However, running through the same 
		argument with a sequence furnished by \eqref{eq:phi 2}  shows that $\mu(\{1/2\})=\infty$ -- a contradiction.\\

		So, if $w_n$ is given by the reciprocal of any one of the three aforementioned multiplicative functions, then it can't come from any measure as in Equation \eqref{eq:McCarthy's weights}.\\

		We now move to the second claim of Theorem \ref{thm:other isometries} -- that in each case, there is a $\delta$ such that $\mathfrak{M}$ is isometrically isomorphic to $H^\infty(\Omega_\delta)\cap\mathcal D$. Let's first examine the case where 
		$w_n=\frac{1}{d(n)}$. In this case, $d(n)=o(n^\epsilon)$ for each positive $\epsilon$ (see \cite{Ha}). So, for each $\sigma>0$, $w_n^{-1}n^{-2\sigma}\to 0$ as $n\to\infty$, and we can certainly find a $C_\sigma$ such that 
		$w_n^{-1} n^{-2\sigma} \le C_\sigma$ for all $n$. The other conditions in Remark \ref{rm:3 conditions} are satisfied: $w_n$ is clearly multiplicative and $w_{p^{k}}=\frac{1}{d(p^{k})}=\frac{1}{k+1}$ is decreasing with $k$ for each prime $p$. 
		Therefore,
			\[\mathfrak{M}\equiv H^\infty(\Omega_0)\cap\mathcal D.\]
	
		If instead $w_n=\frac{1}{\phi(n)}$, then the conditions of of Remark \ref{rm:3 conditions} are satisfied with $\delta=\frac{1}{2}$  as from \eqref{eq:phi 1}: There is some $K>0$ such that $\frac{\phi(n)}{n}\leq K$, which implies that
			\[\phi(n)n^{-1}=\phi(n)n^{-2(1/2)}<K\]
		for each $n$, $w_n$ is multiplicative and $w_{p^k}\leq p^{-1}w_{p^{k-1}}$. Thus
			\[\mathfrak{M}\equiv H^\infty(\Omega_{1/2})\cap\mathcal D.\]

		Finally, suppose that $w_n=\frac{1}{\sigma(n)}$. To see that the conditions of Remark \ref{rm:3 conditions} are satisfied when $\delta=\frac{1}{2}$, note that $w_n$ is multiplicative and $w_{p^k}=\frac{p-1}{p^{k+1}-1}\leq p^{-1}\frac{p-1}
		{p^k-1}$ for each $k$ and each prime $p$. Now, for each $\epsilon>0$, \eqref{eq:sigma} implies the existence of some postive number $K_\epsilon$ such that $\frac{\sigma(n)}{n^{1+\epsilon}}\leq K_\epsilon$ for each $n$, so that
			\[\sigma(n)n^{-2\left(\frac{1}{2}(1+\epsilon)\right)}<K_\epsilon.\]
		Hence
			\[\mathfrak{M}\equiv H^\infty(\Omega_{1/2})\cap\mathcal D.\]
	\end{proof}

\end{document}